\documentclass[12pt]{article}

\usepackage[a4paper]{geometry}
\usepackage[colorlinks=true,citecolor=black,linkcolor=black,urlcolor=blue]{hyperref}
\usepackage{amsmath,amssymb,amsthm}
\usepackage{bbm}
\usepackage{enumerate}
\usepackage{ytableau}
\usepackage{amsfonts}
\usepackage{graphicx}
\usepackage{caption}
\usepackage{subfig}
\usepackage{thmtools}
\usepackage{thm-restate}
\usepackage{cleveref}

\def\Ker{\operatorname{Ker}}

\def\Hilb{\operatorname{Hilb}}
\def\Hom{\operatorname{Hom}}

\def\S{\mathcal{S}}
\def\W{\mathcal{W}}
\def\a{\mathbf{a}}
\newcommand{\R}{\mathcal{R}}
\newcommand{\C}{\mathbb{C}}
\DeclareMathOperator{\im}{Im}
\DeclareMathOperator{\Span}{Span}
\providecommand{\href}[2]{#2}

\newtheorem{lemma}{Lemma}[section]
\theoremstyle{definition}
\newtheorem{definition}{Definition}[section]

\author{Daniel Mulcahy\thanks{Supported by The Hamilton Mathematical Trust.}\\
\small School of Mathematics\\[-0.8ex]
\small Trinity College Dublin\\[-0.8ex] 
\small Dublin, Ireland\\
\small\tt mulcahda@tcd.ie\\
\and
Peter Phelan\thanks{Supported by The Hamilton Mathematical Trust.}\\
\small School of Mathematics\\[-0.8ex]
\small Trinity College Dublin\\[-0.8ex]
\small Dublin, Ireland\\
\small\tt phelanpe@tcd.ie }

\title{Counting Dimensions of Tangent Spaces to Hilbert Schemes of Points}
\date{\today}

\begin{document}
\maketitle
\begin{abstract}
In this paper we prove two results which further classify smoothness properties of Hilbert schemes of points. This is done by counting classes of arrows on Young diagrams corresponding to monomial ideals, building on the approach taken by Jan Cheah to show smoothness in the 2 dimensional case. We prove sufficient conditions for points to be smooth on Hilbert schemes of points in 3 dimensions and on nested schemes in 2 dimensions in terms of the geometry of the Young diagram. In particular, we proved that when the region between the two diagrams at a point of the nested scheme is rectangular, the corresponding point is smooth. We also proved that if the three dimensional Young diagram at a point can be oriented such that its horizontal layers are rectangular, then the point is smooth.
\end{abstract}

\newpage

\section{Introduction}
The Hilbert scheme of $n$ points in $d$ dimensions, denoted $\Hilb^n(\C^d)$ is an $nd$-dimensional scheme parametrising all ideals of codimension $n$ in the polynomial ring $\C[x_{1}, \dots, x_{d}]$.
\[\Hilb^n(\C^d) = \{I \subset \C[x_1, \dots, x_d] \mid \dim_{\C} \C[x_1, \dots, x_d]/I = n\}\]
In general $\Hilb^n(\C^d)$ is not smooth and we are interested in characterizing smooth points and singularities.

In this paper, we shall use methods based on counting arrows in visual diagrams to prove two results providing sufficient
conditions for smoothness of points on certain variants of Hilbert schemes.
\\

First we shall see in section \ref{torus} that, by considering the action of the algebraic torus group rescaling variables of
$\C[x_1,\dots,x_n]$, we may focus our investigation specifically on the \emph{monomial} ideals in $\Hilb^{n}(\C^d)$.
In particular, if every monomial ideal is a smooth point, then every point is smooth.

Let $I$ be a monomial ideal in $R = \mathbb{C}[x_1, \dots x_d]$. The monomials of $R\setminus I$ form a basis of $R/I$.
With axes corresponding to exponents of $x_i$, these monomials can be represented as a Young diagram (or its higher dimensional analogue),
i.e. the presence of one square implies the presence of all those nearer to the origin.
In section \ref{firstsec} we shall see that the tangent space at $I$ is isomorphic to $\Hom(I,R/I)$ and show how in $2$ dimensions
a basis of that space can be represented by arrows on the Young diagram. We shall then demonstrate how this representation
can be used to easily count the dimension of the tangent space, reproducing the known result that $\Hilb^{n}(\mathbb{C}^{2})$ is smooth for every $n$.

In section \ref{sec2} we look at the nested Hilbert scheme of $(n,m)$ points for $n \geq m$ in two dimensions, denoted $\Hilb^{n,m}(\mathbb{C}^{2})$. This is a $2n$-dimensional scheme which parametrises pairs of nested ideals of codimension $n$ and $m$ respectively.
$$ \Hilb^{n,m}(\mathbb{C}^{2}) = \{ (I,J) \ |\ I \subset J,  \ \dim \mathbb{C}[x,y]/I = n, \ \dim \mathbb{C}[x,y]/J = m\} $$
Using a generalization of the method used by Cheah \cite[2.6]{nested_schemes} to prove smoothness in the $m=n-1$ case,
we shall prove the following by considering arrows on the two Young diagrams:
\begin{restatable}{thm}{nested}
	\label{thm:nested} Let $(I,J)$ be a point on $\Hilb^{n,m}(\mathbb{C}^2)$ such that $I,\ J$ are monomial ideals and the monomials of $J \setminus I$ form a rectangle $\R$ in the Young diagram representation. Then the tangent space at $(I,J)$ has dimension $2n$.
\end{restatable}

In section $4$ we shall consider the three dimensional case $ \Hilb^{n}(\mathbb{C}^{3})$.
This is known not to be smooth (for $n \geq 4$), but we are interested in determining which monomial ideals are singular.
Considering the three dimensional version of a Young diagram, we shall use an analogue of the proof in $2$ dimensions to prove that
whenever all horizontal (up to interchanging variables/axes) layers of the Young diagram are rectangular, the corresponding ideal is a smooth point.
This geometric condition on the Young diagram can be restated in terms of the generators:
\begin{restatable}{thm}{cubic} 
	\label{thm:cubic} Let $I \subset R := \mathbb{C}[x,y,z]$ be an ideal generated by monomials of the form $x^iz^j$ and $y^iz^j$
	such that $R/I$ is a vector space of dimension $d$.
	Then $\dim \Hom_R(I,R/I)=3d$.
\end{restatable}

\section{Singularities on Hilbert Schemes of Points}
\label{firstsec}
\subsection{Action of the Algebraic Torus}
\label{torus}
The algebraic torus group $T = (\mathbb{C}^{*})^{d} $ acts on $\C[x_1,\dots, x_d]$ by rescaling variables.
\[(t_1,\dots,t_d)\cdot f(x_1,\dots,x_d)=f(t_1x_1,\dots,t_dx_d)\] This action preserves the colength of ideals, inducing an action on $\Hilb^{n}(\mathbb{\C}^{d})$.
For a compact normal variety $X$ with an action by $ T $, if the subset $ X^{T} \subset X $ of fixed points consists of smooth points, then $ X $ is a smooth variety; when $X$ is not smooth the dimensions of tangent spaces at the torus fixed points give a natural way to measure non-smoothness.
The fixed points of $T$ on $\Hilb^{n}(\C^d)$ are precisely the monomial ideals of codimension $n$, hence only monomial ideals are required to determine smoothness.
The details of this argument can be found in \cite{schemes}.

From \cite{nested_schemes} we know that the tangent space of $ \Hilb^{n}(\mathbb{\C}^{d}) $ at a point $I$ is isomorphic to $ \Hom_{R}(I,R/I) $.
Additionally with the action of $ (\mathbb{C}^{*})^{d} $ lifted to $ \Hom_{R}(I,R/I) $, we may find a weight basis by considering only those elements $ \mathcal{W} \subset \Hom_{R}(I,R/I) $ of pure weight, i.e. morphisms sending each generator of $I$ to $0$ or an element of the basis $\S$ of $R/I$ represented by the monomials of $R\setminus I$.

\subsection{Young Diagram Representation}
We shall represent monomials $x_{1}^{i_{1}} \dots x_{d}^{i_{d}}$ by the corresponding boxes at $(x_i)$ on the $x_{1}, \dots, x_{d}$ hyperplane.
This hyperplane is divided into three regions: the region $\S$ of monomials in $R \setminus I$ which forms a Young diagram, the region of monomials in $I$ we call $\mathcal I$, and
the region we call $\mathcal{Z}$ below the $x_{i}$ axes, i.e. consisting of $x_{1}^{i_{1}} \dots x_{d}^{i_{d}}$ where some $i_k <0$.
\begin{definition}
	An \emph{arrow} $\alpha \to \beta$ on the Young diagram is a pair $(\alpha,\beta) \in \mathcal{I} \times (\S \cup \mathcal{Z})$, visually represented on the diagram as an arrow pointing from $\alpha$ to $\beta$.
\end{definition}
\begin{definition}
	Given some arrow $\mathbf{a} = \alpha \to \beta$, we may obtain another arrow $\alpha' \to \beta'$ by simultaneously moving the head and tail 1 block in any direction as long as $\alpha' \in \mathcal{I} $ and $\beta' \in \S \cup \mathcal{Z}$. If we repeat this step several times to obtain an arrow $\mathbf{a}'$, then we say $\mathbf{a}$ is \emph{dragged to} $\mathbf{a}'$.
\end{definition}
We may assume that such a dragging happens along the boundary of $\S$ i.e. keeping the tail directly or diagonally adjacent to a block in $\S$.
Let $ \{ \alpha_{i} \}_{i \in J} $ be the canonical generators of $I$.
A morphism $f \in \mathcal{W}$ can be viewed as a collection of arrows $\alpha_i \to f(\alpha_i)$ for those $f(\alpha_{i}) \neq 0$.
Suppose that for two generators $ \alpha_{i}, \alpha_{j}$ there exist monomials $ m,n $ such that $ m \cdot \alpha_{i} = n \cdot \alpha_{j}$;
then since $f$ is an $R$-module homomorphism this implies that $ m \cdot f(\alpha_{i}) = n \cdot f(\alpha_{j}) $.
Hence for any monomial $m$ and $\beta \in \S$, if $m\beta \neq 0$ (in $R/I$) then any morphism containing the arrow $\alpha_i \to \beta$ must also contain the arrow $\alpha_j \to m\beta /n$.
In particular, if $m\beta/n$ contains negative powers of $x$ or $y$, the arrow $\alpha_i \to \beta$ can never occur in a morphism. 
Geometrically this means that if an arrow from $\alpha_i$ can be dragged to $\alpha_j$,
then those arrows must appear together in any morphism, and if the head of one is beyond some axis then the other can never appear.
Any collection of arrows from generators to $S$ satisfying these properties represents an element of $\W$.
For some arrow $\a$ from $\alpha_i$ to $\beta \in \S$, either there exists some morphism $\langle \a \rangle \in \W$ containing $\a$
that sends the largest possible number of generators to $0$, and thus consists only of arrows that $\a$ can be dragged to, or no such morphism exists in which case we say $\langle\a\rangle=0$.
Then $\W$ is generated by the set $\mathcal{T}$ of morphisms of the form $\langle\a\rangle$, and $\mathcal{T}$ is a basis for $\Hom(I,R/I)$. \\

\subsection{2-Dimensional Case}
Here we present a restatement of the proof from \cite{nested_schemes} in two dimensions, to introduce the type of approach we shall take to obtain our own results.

In the two dimensional case, there are three different kinds of arrows distinguished by which interval their direction (calling positive $x$ direction right, and positive $y$ up) belongs to:
\begin{enumerate}
	\item Arrows with directions in [left, up).
	\item Arrows with directions in (down, left).
	\item Arrows with directions in (right, down]
\end{enumerate}
We label the generators $\alpha_{0} \dots \alpha_{m} $ from bottom to top.
$$
\begin{array}{c}
\begin{ytableau}
\none[\alpha_{3}] & \none & \none & \none & \none  \\
&  & \none  &  \none & \none & \none    \\
&  & \none[\alpha_{2}]  &  \none & \none & \none   \\
&  &  &  &  &  \none[\alpha_{1}]   \\
&  &  &  &  &  &  \none[\alpha_{0}]  
\end{ytableau}\\
\\
\S
\end{array}
$$
\begin{lemma}
	Arrows of the second kind can never occur in a morphism.
\end{lemma}
\begin{proof}
	These arrows can always be dragged around corners so can be dragged to $\alpha_0$ where their head will be below the axis.
\end{proof}
Let $ p_{i} $ be the vertical distance between $ \alpha_{i}$ and $\alpha_{i+1} $, and be $ q_{i} $ the horizontal distance between $ \alpha_{i}$ and $\alpha_{i-1} $.
For some generator $ \alpha_{i}$, define $ P_{\alpha_{i}} = \{ \beta \in B \ | \ \beta \ \text{lies to the left of} \ \alpha_{i}, \ y^{p_{i}} \beta \in I \} $ and $ Q_{\alpha_{i}} = \{ \beta \in B \ | \ \beta \ \text{lies lower than } \ \alpha_{i}, \ x^{q_{i}} \beta \in I \} $.
The morphisms in $\mathcal{T}$ with arrows of the first kind are $ f_{\alpha,\beta} = \langle \alpha \to \beta\rangle$ where $\beta \in P_{\alpha} $.
In this case $f_{\alpha,\beta}$ takes all generators above $\alpha$ to $0$.
Similarly arrows of the third kind generate morphisms $ f_{\alpha,\beta} $ such that $\beta \in Q_{\alpha} $.

As in \cite{nested_schemes}, it suffices to show that $$ | \mathcal{T} | = \sum_{\alpha}|P_{\alpha}| + \sum_{\alpha}|Q_{\alpha}| = n + n = 2n $$ 
To count the elements of $ P_{\alpha_{i}}$, we count all squares within the vertical distance $ p_{i} $ of the top of each column, which is the same as counting every square within the vertical interval from $\alpha_i$ (inclusive) to $\alpha_{i+1}$.
This amounts to counting every square in every vertical interval of the diagram, giving $n$ unique arrows.
Considering instead the rows and horizontal distances $ q_{i} $, the same argument shows also that $ \left|Q_{\alpha}\right|=n $.  \\
For example, consider $ I = (x^{6},x^{5}y,x^{2}y^{2},y^{4}) = (\alpha_{0},\alpha_{1},\alpha_{2},\alpha_3) $.
$$
\begin{array}{ccc}
\begin{ytableau}
\none[\alpha_{3}] & \none & \none & \none & \none  \\
\beta  & \beta & \none  &  \none & \none & \none    \\
&  & \none[ \alpha_{2} ]  &  \none & \none & \none   \\
&  & \beta & \beta & \beta &  \none[ \alpha_{1} ]   \\
&  &  &  &  & \beta &  \none[ \alpha_{0} ]  
\end{ytableau} &
\begin{ytableau}
\none[ \alpha_{3} ] & \none & \none & \none & \none  \\
\beta  & \beta & \none  &  \none & \none & \none    \\
&  & \none[ \alpha_{2} ]  &  \none & \none & \none   \\
&  & \beta & \beta & \beta & \none[ \alpha_{1} ]   \\
&  &  &  &  &  &  \none[ \alpha_{0} ]  
\end{ytableau} &
\begin{ytableau}
\none[ \alpha_{3} ] & \none & \none & \none & \none  \\
\beta & \beta & \none  &  \none & \none & \none    \\
\beta & \beta & \none[ \alpha_{2} ]  &  \none & \none & \none  \\
&  &  &  &  & \none[ \alpha_{1} ]  \\
&  &  &  &  &  &\none[ \alpha_{0} ]  
\end{ytableau}\\ \\
P_{\alpha_{0}} & P_{\alpha_{1}} & P_{\alpha_{2}} \end{array} $$
Counting the boxes $\beta$ for each $P_{\alpha_i}$ is equivalent to counting the vertical intervals
$$
\ytableausetup{boxsize=2em}
\begin{ytableau}
\none[\alpha_{3}] & \none & \none & \none & \none  \\
P_{\alpha_{2}} & P_{\alpha_{2}} & \none  &  \none & \none & \none    \\
P_{\alpha_{2}} & P_{\alpha_{2}} & \none[ \alpha_{2} ]  &  \none & \none & \none  \\
P_{\alpha_{1}} & P_{\alpha_{1}} & P_{\alpha_{1}} & P_{\alpha_{1}} & P_{\alpha_{1}} & \none[ \alpha_{1} ]   \\
P_{\alpha_{0}} & P_{\alpha_{0}} & P_{\alpha_{0}} & P_{\alpha_{0}} & P_{\alpha_{0}} & P_{\alpha_{0}} &  \none[ \alpha_{0} ]  
\end{ytableau}$$
\ytableausetup{boxsize=1.5em}
Similarly the $Q_{\alpha_i}$
$$
\begin{array}{ccc}
\begin{ytableau}
\none[ \alpha_{3} ] & \none & \none & \none & \none  \\
\beta & \beta & \none  &  \none & \none & \none    \\
\beta & \beta & \none[ \alpha_{2} ]  &  \none & \none & \none   \\
&  &  & \beta & \beta &  \none[ \alpha_{1} ]   \\
&  &  &  & \beta & \beta &  \none[ \alpha_{0} ]  
\end{ytableau}&
\begin{ytableau}
\none[ \alpha_{3} ] & \none & \none & \none & \none  \\
&  & \none  &  \none & \none & \none    \\
&  & \none[ \alpha_{2} ]  &  \none & \none & \none   \\
&  & \beta & \beta & \beta & \none[ \alpha_{1} ]   \\
&  &  & \beta & \beta & \beta &  \none[ \alpha_{0} ]  
\end{ytableau}&
\begin{ytableau}
\none[ \alpha_{3} ] & \none & \none & \none & \none  \\
&  & \none  &  \none & \none & \none    \\
&  & \none[ \alpha_{2} ]  &  \none & \none & \none  \\
&  &  &  &  & \none[ \alpha_{1} ]  \\
&  &  &  &  & \beta &\none[ \alpha_{0} ]  
\end{ytableau}\\
Q_{\alpha_{3}} & Q_{\alpha_{2}} & Q_{\alpha_{1}} 
\end{array} $$
can be counted as
$$
\ytableausetup{boxsize=2em}
\begin{ytableau}
\none[\alpha_{3}] & \none & \none & \none & \none  \\
Q_{\alpha_{3}} & Q_{\alpha_{3}} & \none  &  \none & \none & \none    \\
Q_{\alpha_{3}} & Q_{\alpha_{3}} & \none[ \alpha_{2} ]  &  \none & \none & \none  \\
Q_{\alpha_{3}} & Q_{\alpha_{3}} & Q_{\alpha_{2}} & Q_{\alpha_{2}} & Q_{\alpha_{2}} & \none[ \alpha_{1} ]   \\
Q_{\alpha_{3}} & Q_{\alpha_{3}} & Q_{\alpha_{2}} & Q_{\alpha_{2}} & Q_{\alpha_{2}} & Q_{\alpha_{1}} &  \none[ \alpha_{0} ]  
\end{ytableau}$$
\ytableausetup{boxsize=1.5em}

\section{Nested Hilbert Schemes}
\label{sec2}
We shall now apply the concepts discussed above to the case of nested Hilbert schemes
$$ \Hilb^{n,m}(\mathbb{C}^{2}) = \{ (I,J) \ |\ I \subset J,  \ \dim \mathbb{C}[x,y]/I = n, \ \dim \mathbb{C}[x,y]/J = m\}. $$
Let $I \subset J$ be monomial ideals.
Recall that the tangent spaces of $\Hilb^{n}(\mathbb{C}^{2}) $ and $\Hilb^{m}(\mathbb{C}^{2}) $ are isomorphic to $\Hom_{R}(I,R/I)$ and $\Hom_{R}(J,R/J)$ respectively.
Also note that for $I \subset J$, the natural embedding $ I \hookrightarrow J $ and projection $ R/I \to R/J $
induce the composition maps 
$$ \phi: \Hom_{R}(J,R/J) \to \Hom_{R}(I,R/J), \ \psi: \Hom_{R}(I,R/I) \to \Hom_{R}(I,R/J).$$
Now we define $$ (\phi - \psi):\Hom_{R}(I,R/I) \oplus \Hom_{R}(J,R/J) \to \Hom_{R}(I,R/J) $$ $$ (\phi - \psi)(f_{1},f_{2}) = \phi(f_{1}) - \psi(f_{2}) $$ From \cite[0.4]{nested_schemes} we know that the pullback $\Ker(\phi - \psi) $ of the maps $ \phi, \psi $ is isomorphic to the tangent space of $ \Hilb^{n,m}(\mathbb{C}^{2}) $ at $(I,J)$.

Given an arrow $\mathbf a=\alpha \to \beta$, where $\alpha$ is a generator of $I$ and $\beta$ a monomial in $R\setminus I$, we define the morphism $\langle \mathbf a \rangle_I$
(resp. $\langle \mathbf a \rangle_{I,J}$) generated by $\mathbf a$ in $\Hom(I,R/I)$ (resp. $\Hom(I,R/J)$) to be the morphism that takes the largest number of canonical generators of $I$ to $0$; if no morphism contains the arrow we say $\langle \mathbf a \rangle_I=0$
(resp. $\langle \mathbf a \rangle_{I,J}=0$). Denote $\alpha,\beta$ by $t(\mathbf a), h(\mathbf{a})$ respectively.

\nested*
\begin{proof}
	Label the generators of $I$ found above the rectangle counting downwards $\alpha_0, \dots, \alpha_k$ and those below counting upwards $\beta_0, \dots, \beta_l$
	The tangent space is isomorphic to $\Ker (\psi \oplus 1 - 1 \oplus \phi) \subset \Hom(I,R/I)\oplus \Hom(J,R/J)$, where $\psi : \Hom(I,R/I) \to \Hom(I,R/J)$ and $\phi : \Hom(J,R/J) \to \Hom(I,R/J)$
	are induced by $I \hookrightarrow J$, $R/I \twoheadrightarrow R/J$.
	The Young diagram of monomials $S_I$ in $R/I$ contains the monomials $S_J$ of $R/J$. Call the boundary curve of $S_I$ (extending along the axes as shown in figure \ref{fig:bounds}) $B_I$, and that of $S_J$ $B_J$.
	We say the region on the same side of one of these boundary curves as the origin is below it, and the region on the other side is above it.
	Morphisms in $\Hom(I,R/J)$ are those whose arrows $A$ satisfy the condition that, if an arrow in $A$ can be dragged to an arrow from another generator while keeping its head below $B_J$ and tail above $B_I$,
	then that arrow must be in $A$.
	It suffices to consider only the morphisms $W$ consisting of arrows of a single length and direction.
	First we show that $\psi-\phi$ is surjective.
	\begin{figure}[ht]%
		\centering
		\subfloat[$B_I$]{{\includegraphics[width=0.4\textwidth]{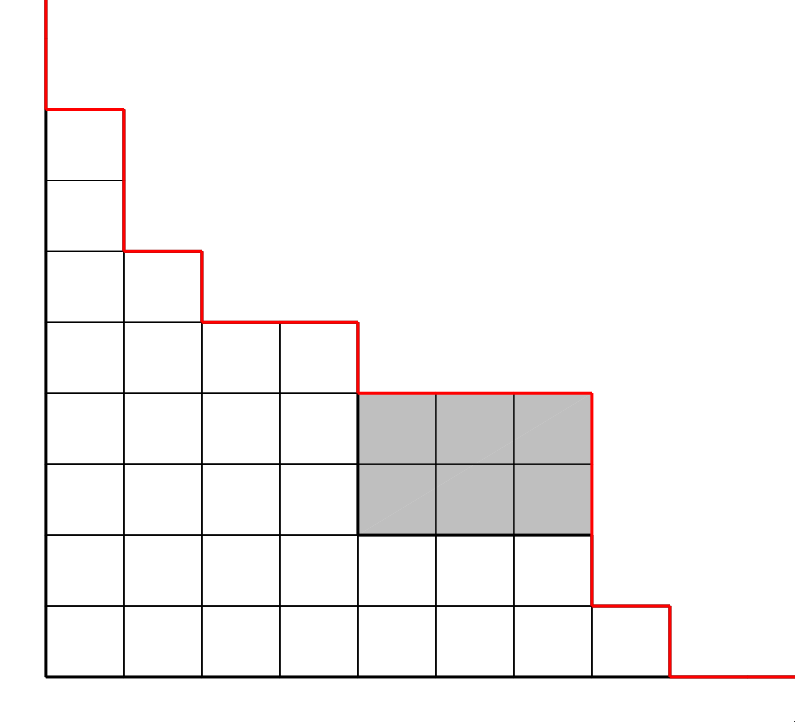} }}%
		\qquad
		\subfloat[$B_J$]{{\includegraphics[width=0.4\textwidth]{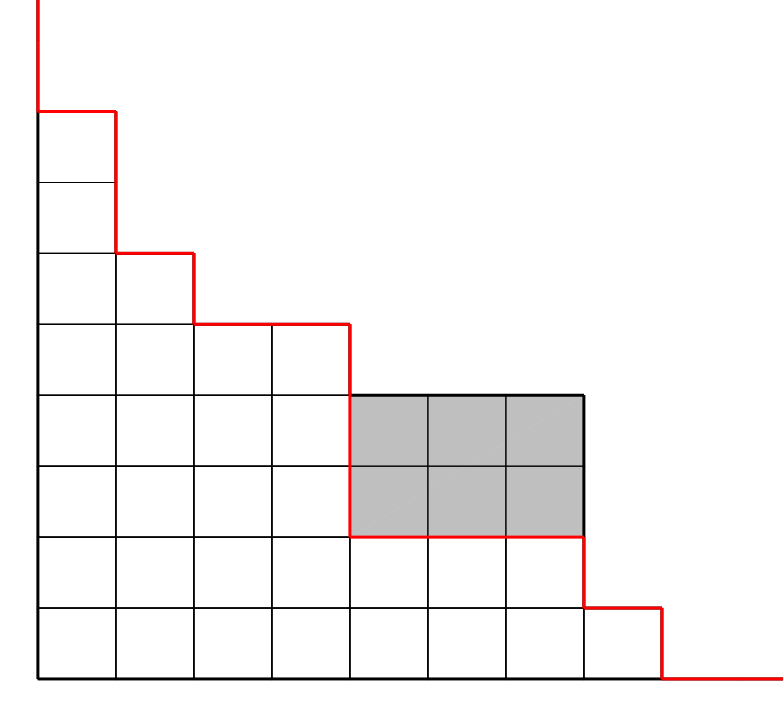} }}%
		\caption{the boundary curves}%
		\label{fig:bounds}%
	\end{figure}
	
	Morphisms in $\im \psi$ are those whose arrows $A$ satisfy the condition that,
	if an arrow in $A$ can be dragged to another generator keeping its head below and tail above $B_I$ (the head may go above $B_J$) then that arrow must be in $A$.
	If an arrow is dragged in this manner to another but cannot be dragged while keeping its head below $B_J$, then we say it is dragged through the rectangle $\R$.
	The remaining morphisms are those that do not satisfy this condition, i.e. they contain some arrow $\mathbf a$ that could be dragged through the rectangle to $\mathbf b$, but do not contain $\mathbf b$.
	We may assume $h(\mathbf a), h(\mathbf b)$ are directly to the left of or below $\R$.
	Consider such morphisms where $\mathbf a$ points downwards from the $\alpha$ generators and which are generated by $\mathbf a$. It suffices to consider the subset $G$ for which $h(\mathbf a)$ is to the left of $\R$:
	if it is below $\R$ and $h(\mathbf b)$ is to the left, then $\langle\mathbf b\rangle_{I,J} \in G$ and $\langle\mathbf a\rangle_{I,J}+\langle\mathbf b\rangle_{I,J} \in \im \psi$.
	Depending on the position of $\R$ in the Young diagram, $\alpha_k$ may lie directly above $\R$, and in this case no $f \in G$ 
	can contain an arrow from $\alpha_k$ as it would point down and to the left. Otherwise $\alpha_k$ is a canonical generator of
	$J$, as are $\alpha_0, \dots, \alpha_{k-1}$ in either case. Thus it is clear that any $f \in G$ can also be considered as
	a morphism in $\Hom(J, R/J)$ with the same arrows, so $f \in \im \phi$.
	The same analysis applies for the morphisms $G'$ obeying the mirrored condition ($\langle \mathbf a \rangle_{I,J}$ where $\mathbf a$ points upwards from the $\beta$ generators to directly below $\R$ and could be dragged up through $\R$), showing $G' \subset \im \phi$.
	
	From our definition of $G$ and $G'$, we have $$\Hom(I,R/J) = \im \psi \oplus \Span\{ G\} \oplus \Span\{ G'\}$$ and thus $\psi - \phi$ is surjective.
	\begin{figure}[ht]
		\begin{minipage}{.5\textwidth}
			\centering
			\includegraphics[width=0.8\textwidth]{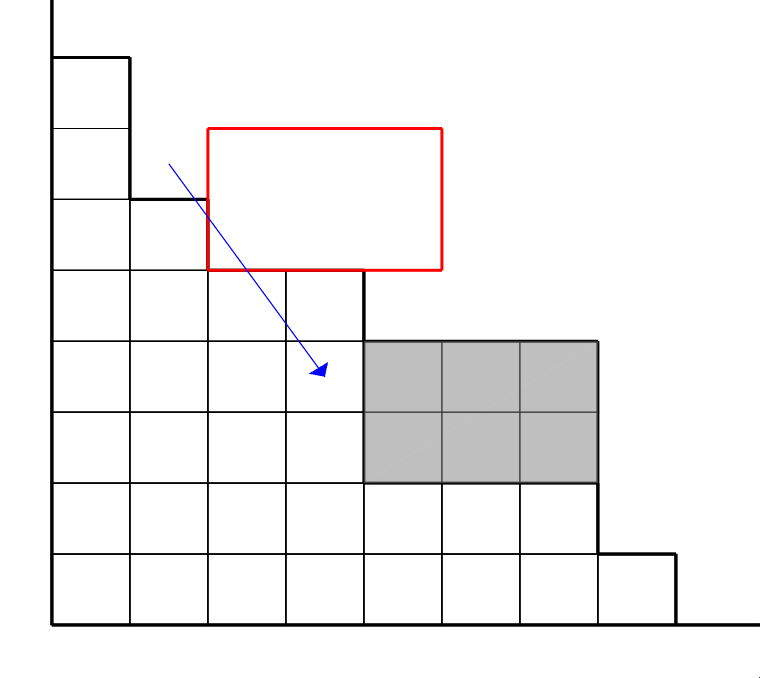}
			\caption{a $G$-type rectangle}
			\label{fig:cokrects}
		\end{minipage}
		\begin{minipage}{.5\textwidth}
			\centering
			\includegraphics[width=0.8\textwidth]{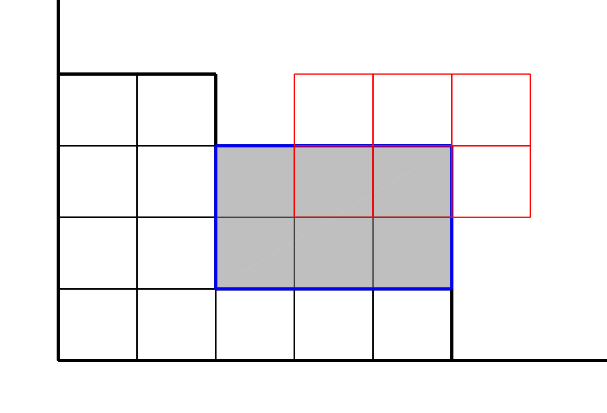}
			\caption{extra $G$-type rectangles}
			\label{fig:poss}
		\end{minipage}
	\end{figure}
	
	Consider the rectangle $\R'$ obtained by moving $\R$ back along $\mathbf a$ such that the position of $\R'$ with respect to $t(\mathbf a)$ is the same as the position of $\R$ with respect to $h(\mathbf a)$, as in figure \ref{fig:cokrects}.
	Dragging the arrow through $\R$ is equivalent to dragging the tail of the arrow above $B_I$ through $\R'$.
	The requirement that $\mathbf a$ cannot be dragged to $\mathbf b$ without going through $\R$ means that the bottom left corner $C_{\R'}$ cannot be above $B_I$.
	The requirement that $h(\mathbf a)$ be directly to the left of $\R$ and the $h(\mathbf b)$ be directly below means that the top side $T_{\R'}$ and right side $R_{\R'}$ of $\R'$ must be above $B_I$.
	In fact, every rectangle $\R'$ (of equal size to $\R$) satisfying these three properties that occurs higher than $\R$
	corresponds uniquely in this manner to an arrow $\mathbf a$ (with $t(\mathbf a)$ the first generator to the left of $\R'$) such that
	$\langle\mathbf a\rangle_{I,J} \in G$ unless $\langle\mathbf a\rangle_{I,J}=0$. We call these rectangles $G$-type.
	$\langle\mathbf a\rangle_{I,J}=0$ only if $h(\mathbf a)$ is to the left of $t(\mathbf a)$, which only occurs for the $n-m$ rectangles $\R'$ for which $C_{\R'}$ lies in $\R$ (including on the boundary)
	i.e. the bottom-left-most square of $\R'$ is one of those shown in red in figure \ref{fig:poss}.
	
	\begin{figure}
		\centering
		\includegraphics[width=0.4\textwidth]{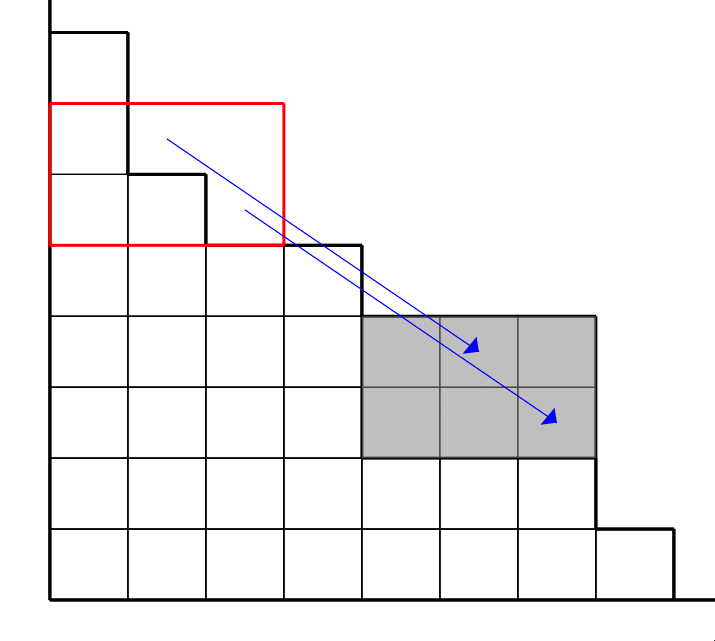}
		\caption{a kernel-type rectangle}
		\label{fig:kerrects}
	\end{figure}
	Morphisms in $\Ker \psi$ are those consisting solely of arrows with heads in $\R$, and are generated by $\langle \mathbf a \rangle_I$ where $\mathbf a$ cannot be dragged out of $\R$.
	Dragging $\R$ back along $\mathbf a$ (or equivalently any other arrow of $\langle \mathbf a \rangle_I$) gives a rectangle $\R'$
	such that $T_{\R'}$ and $R_{\R'}$ intersect $B_I$ (as $t(\mathbf a)$ cannot be dragged out of $\R'$) and the upper-right corner $D_{\R'}$ is above $B_I$ (as $t(\mathbf a)$ is contained in $\R'$).
	Any rectangle $\R'$ with these properties (kernel-type) corresponds uniquely in this fashion to a morphism in $\Ker \psi$ consisting of arrows with tails at each generator found in $\R'$.
	Every such rectangle must occur either lower or farther left than $\R$, so we may first consider only those rectangles to the left.
	
	We claim that the number of $G$-type rectangles at a given height above $\R$ is equal to the number of kernel-type rectangles at that height (above $\R$).
	For our purposes a \emph{height} $h$ is a specific vertical interval of the same size as the rectangle's height, we label the row of the Young diagram at the top of this interval $\mathcal{T}_h$ and the row below the bottom $\mathcal{B}_h$
	If $\R_1$ is the leftmost $G$-type rectangle and $\R_2$ is the rightmost at a given height, then clearly every rectangle in between is also $G$-type.
	$C_{\R_2}$ must lie on $B_I$, so the bottom left corner of $\R_2$ lies immediately upwards and to the right of the rightmost square in the row below $\R_2$.
	$\R_1$ being leftmost means either the bottom of $R_{\R_2}$ or the left end of $T_{\R_2}$ is 1 square to the right of $B_I$ (whichever occurs farther right), as shown in figure \ref{fig:cokt}.
	In the first case, the distance from $\R_1$ to $\R_2$ is the width $w$ of $\R$; otherwise it is the horizontal distance $d$ from the end of $\mathcal{T}_h$ to the end of $\mathcal{B}_h$.
	Thus the number of $G$-type rectangles at this height is $\min \{d,w\}$ \\
	\begin{figure}[ht]%
		\centering
		\subfloat{{\includegraphics[width=0.4\textwidth]{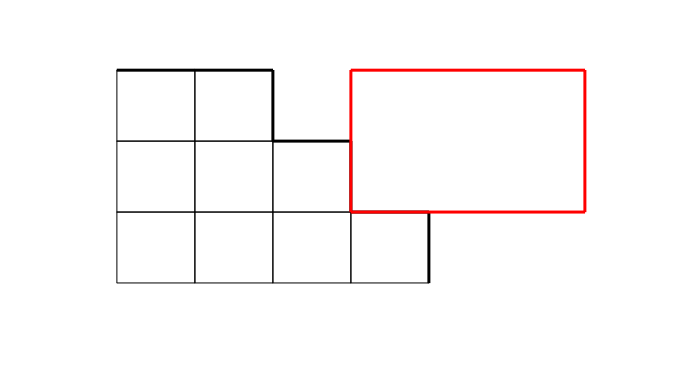} }}%
		\qquad
		\subfloat{{\includegraphics[width=0.4\textwidth]{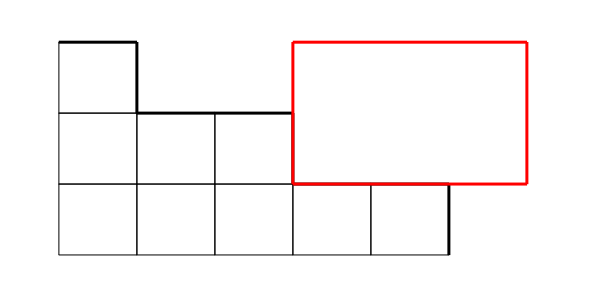} }}%
		\caption{2 cases for leftmost $G$-type rectangle}%
		\label{fig:cokt}%
	\end{figure}
	
	\begin{figure}%
		\centering
		\subfloat{{\includegraphics[width=0.4\textwidth]{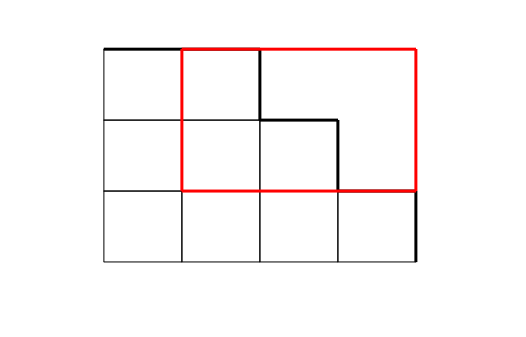} }}%
		\qquad
		\subfloat{{\includegraphics[width=0.4\textwidth]{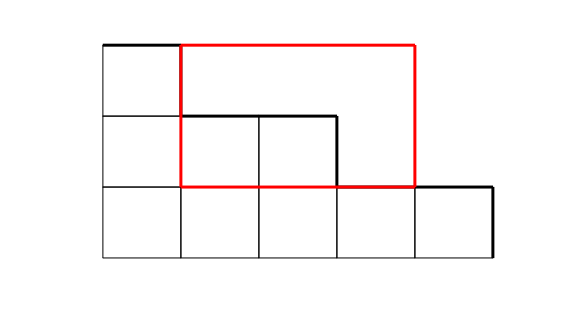} }}%
		\caption{2 cases for rightmost kernel-type rectangle}%
		\label{fig:kert}%
	\end{figure}
	
	If $\R_1'$ is the leftmost kernel-type rectangle and $\R_2'$ is the rightmost at a given height, then clearly every rectangle in between is also kernel-type.
	$D_{\R_1'}$ must be one square to the right of $B_I$, and either the left end of $T_{\R'_2}$ or the bottom of $R_{\R'_2}$ lies on $B_I$ (whichever occurs farther left) as shown in figure \ref{fig:kert}.
	In the first case, the distance from $\R_1$ to $\R_2$ is $w$, in the second it is $d$.
	Thus the number of kernel-type rectangles at this height is $\min \{d,w\}$ as required.
	Mirroring this argument we see the number of $G'$-type rectangles at a given horizontal position is the same as the number of kernel-type rectangles at the same horizontal position (to the right of $\R$).
	$G$-type rectangles excluding the $n-m$ described previously (which are also $G'$-type) are in bijection with morphisms of $G$, and the same is true of $G'$.
	As the kernel-type rectangles give a basis for $\Ker \psi$, we see that $\dim \Ker \psi = \dim \left(\langle G \rangle \oplus \langle G' \rangle\right)+2(n-m)$.
	
	From the exact sequence 
	\[0 \to \Ker (\psi - \phi) \to \Hom(I,R/I)\oplus \Hom(J,R/J) \to \Hom(I,R/J) \to 0\]
	we find 
	\begin{align*}
	\dim \Ker (\psi- \phi) &= \dim \Hom(J,R/J) + (\dim \Ker \psi + \dim \im \psi) - \dim \im \psi \oplus \langle G \rangle \oplus \langle G' \rangle\\
	&= 2m + \dim \Ker \psi - \dim \langle G \rangle \oplus \langle G' \rangle\\
	&= 2m + 2(n-m)\\
	&= 2n
	\end{align*}
\end{proof}
\subsection{Generalisations}
Let $(I,J) \in \Hilb^{n,m}(\mathbb{C}^2)$ monomial ideals and let $\R$ be the shape formed by the monomials of $J/I$ in the Young diagram representation. The above theorem provides a sufficient condition for $(I,J)$ to be a smooth point, but it is not necessary: for $J=(x,y)$ and $I=(x^2,y^2)$, $\Ker(\psi-\phi)$ is spanned by 8 morphisms, $ f^{I}_{x^2,x},f^{I}_{x^2,y},f^{I}_{x^2,xy},f^{I}_{y^2,x},f^{I}_{y^2,y},f^{I}_{y^2,xy},f^{J}_{x,1},f^{J}_{y,1} $ where $f^{I}_{\alpha,\beta} = \langle \alpha \to \beta \rangle_{I} $. $\R$ is not a rectangle here as shown in figure \ref{fig:ex2}. \\
\begin{figure}[h]
	\begin{minipage}{0.5\textwidth}
		\centering
		\includegraphics[width=0.3\textwidth]{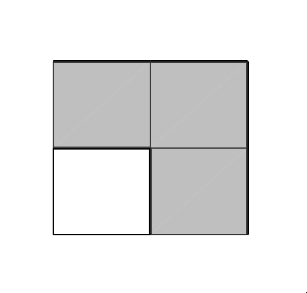}
		\caption{}
		\label{fig:ex2}
	\end{minipage}
	\begin{minipage}{0.5\textwidth}
		\centering
		\includegraphics[width=0.3\textwidth]{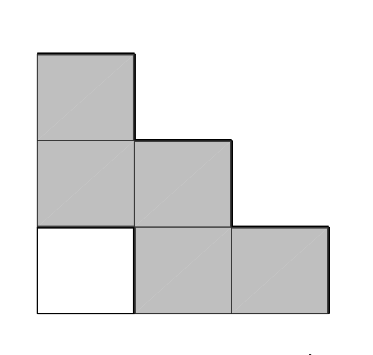}
		\caption{}
		\label{fig:ex1}
	\end{minipage}
\end{figure}
The statement is not true for $\R$ an arbitrary connected shape: we have as a counterexample figure \ref{fig:ex1}, $I=(x^3,x^2y,y^2x,y^3)$ and $J=(x,y)$. Upon computation, we see that $\R$ is connected but the dimension of the tangent space is not $12$.

\newpage

\section{Three Dimensional Case}
Let $R=\mathbb{C}[x,y,z]$, so we now have Young diagrams consisting of cubes in 3 dimensions.

\cubic*
\begin{proof}
	Label the canonical generators in two groups: $\alpha_0=z^p$ downwards in the $xz$ plane to $\alpha_m=x^q$, $\beta_0=\alpha_0$ downwards in the $yz$ plane to $\beta_l=y^r$.
	
	The absence of generators containing nonzero powers of $x$ and $y$ means that if we take the $z$ axis to be vertical, each horizontal layer of blocks in the diagram is rectangular.
	There are three types of arrows possible: firstly those arrows from the generators $\alpha_1, \dots, \alpha_k$ with heads having higher or equal powers of $z$,
	secondly arrows from the generators $\beta_1, \dots, \beta_l$ with the same condition,
	and thirdly all arrows from any generators (except $\alpha_k$, $\beta_l$) with heads having smaller powers of $z$.
	
	\begin{figure}[ht]
		\begin{minipage}{0.5\textwidth}
			\centering
			\includegraphics[width=0.8\textwidth]{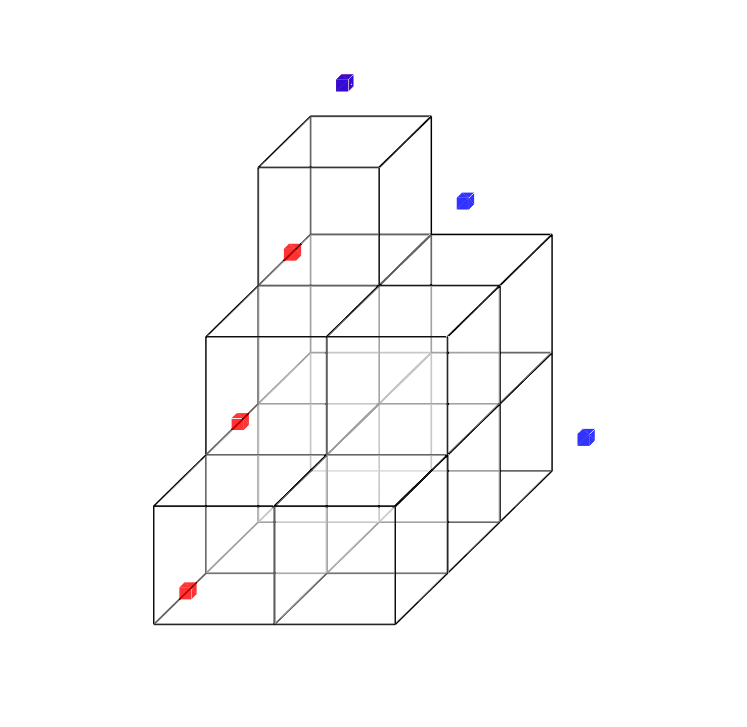}
			\caption{$\alpha$ generators in red and $\beta$ in blue}
			\label{fig:gens}
		\end{minipage}
		\begin{minipage}{0.5\textwidth}
			\centering
			\includegraphics[width=0.8\textwidth]{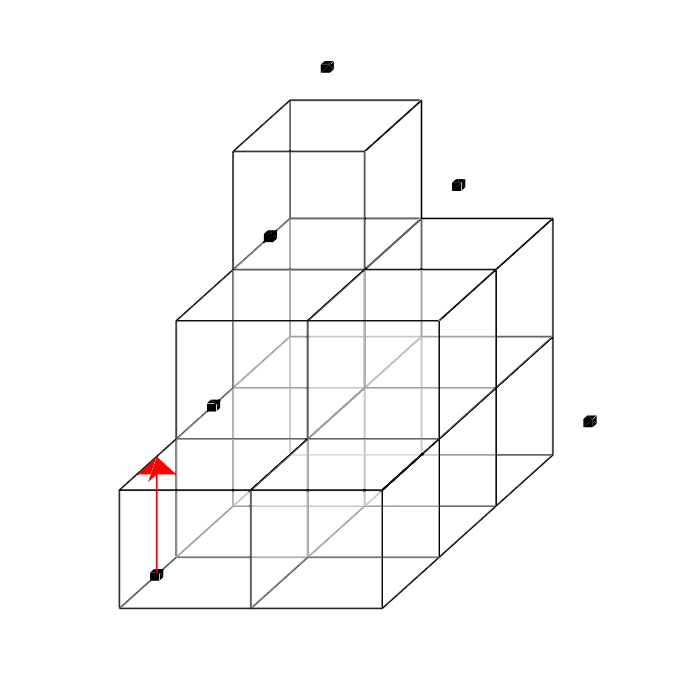}
			\caption{$p_0$}
			\label{fig:p}
		\end{minipage}
	\end{figure}
	
	Consider the first type of arrows.
	For each $\alpha_i$, $i=1, \dots, k$ let us count the morphisms $f$ consisting of these arrows such that $f(\alpha_i) \neq 0$ and $f(\alpha_j)=0$ for $j<i$ and $f$ is generated by an arrow from $\alpha_i$.
	Let $p_i$ be the vertical distance from $\alpha_i$ to $\alpha_{i-1}$.
	$f(\alpha_{i-1})=0$ means that we cannot drag the arrow at $\alpha_i$ to $\alpha_{i-1}$,
	therefore $z^{p_i}f(\alpha_i) \in I$.
	
	Given an arrow $\alpha_i\to \gamma$ of the first type satisfying $z^{p_i}\beta \in I$,
	$\gamma$ has higher or equal powers of $z$ (by arrow type) and $y$ (as $y \nmid \alpha_i$) than $\alpha_i$ so must have a lower power of $x$ as it lies inside the Young diagram.
	Let $\gamma=y^bz^c\alpha_i/x^a$,
	Dragging the arrow will produce arrows with tails at $m \in I$ and heads at $my^bz^c/x^a$.
	Clearly the only way to drag it past the axes is in the negative $x$ direction. This could only be done by dragging it
	to $\alpha_j$ for $j<i$, or to some $\beta_j$. The latter is impossible: $my^bz^c/x^a \in \S \implies m/x^a \in \S$, so
	unless the head is already past the axes, the tail will always be directly to the $x$ direction of a part of the Young diagram
	so cannot be dragged past the corner of any rectangular layer.
	But as $z^{p_i}\beta \in I$, the arrow cannot be dragged sufficiently far in the $z$ direction to reach $\alpha_{i-1}$,
	and the rectangular structure of the diagram means dragging the arrow in the $y$ direction does not reduce this required distance.
	
	Therefore each arrow from $\alpha_i$ to some element of $\S$ with a not lower power of $z$ that is at most $p_i$ blocks below the surface of the diagram
	generates a distinct morphism that is zero on each $\alpha_j, j<i$ and on all $\beta_j$.
	The number of these arrows is the same as the number of boxes in $\S$ contained in the height interval from $\alpha_i$ to $\alpha_{i-1}$,
	i.e.
	those whose powers of $z$ range from $a$ to $a+p_i-1$ if $a$ is the power of $z$ in $\alpha_i$.
	Counting these boxes means that every box will be counted exactly once when we count $\alpha_1, \dots, \alpha_k$, giving $d$ morphisms of this type.
	
	Morphisms corresponding to the second type of arrows are counted the same way interchanging $x$ and $y$, giving $d$ of them.
	
	For each generator (except $\alpha_k, \beta_l$) we count morphisms generated by arrows of the third type which take that generator to an element of $\S$ and all generators with lower powers of $z$ to zero.
	\begin{lemma} If an arrow of this type cannot be dragged to some lower generator directly (i.e.
		without dragging it upwards then down), then it cannot be dragged to a lower generator in any manner.
	\end{lemma}
	\begin{proof}
		Suppose we have such an arrow at a generator $\alpha$ that can be dragged indirectly to a lower generator $\beta$.
		If we call the highest power of $z$ appearing in the tail of any intermediate arrows along the path the height of the path,
		then there exists a path from $\alpha$ to $\beta$ of least height.
		If this height is that of $\alpha$ then we are done.
		If not, we may assume that the tail is dragged up to that height in one direction (assume negative $x$ direction) and down from it in the other (positive $y$),
		as otherwise this part of the path is redundant and the path could be shortened to one unit lower.
		Let $m_1$ and $m_2$ be the first and last values the tail takes
		at this height.
		Then $m_1/(xz)$ and $m_2/(yz)$ appear on the edges of a rectangular layer of the diagram,
		and we may drag the arrow around the border of that rectangle from $m_1/z$ to $m_2/z$ using the fact that the heads of these arrows are in a single rectangular layer.
		This gives us a path of lower height.
	\end{proof}

	Thus for each generator $\alpha$, the morphisms we are counting correspond to arrows downward from $\alpha$ that cannot be dragged any farther down.
	There are two cases to consider: Either $\alpha$ is the only generator at its height, or there is one other generator $\beta$ found on the opposite corner of that layer.
	In the first case, the upward-facing surface of the layer below $\alpha$ (those elements $m$ of $\S$ with power of $z$ one less than $\alpha$ satisfying $zm \in I$)
	is rectangular of $x-$length $\xi$ and $y-$length $\eta$, and arrows can be dragged downward in either the $x$ or $y$ direction.
	Then arrows that cannot be dragged farther down are those for which $x^{\xi}f(\alpha) \in I$ and $y^\eta f(\alpha) \in I$,
	which by the rectangular nature of layers of $\S$ are precisely arrows with heads in the $\xi \times \eta$ rectangle of their layer of $\S$
	found immediately inwards of that layer's far corner, for any layer below $\alpha$.
	This is the same as counting the number of boxes found
	directly below the rectangular surface under $\alpha$.
	
	\begin{figure}[ht]
		\begin{minipage}{0.5\textwidth}
			\centering
			\includegraphics[width=0.8\textwidth]{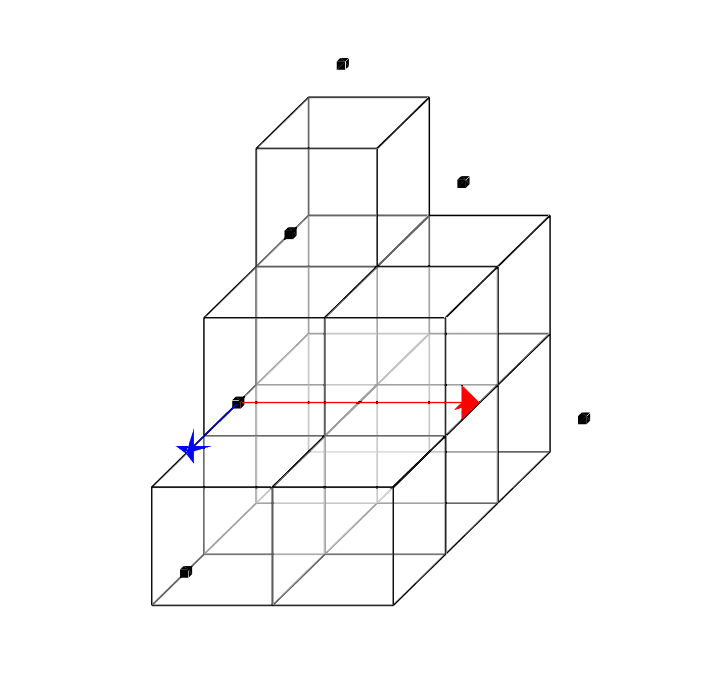}
			\caption{$\xi$ marked in blue, $\eta$ in red}
			\label{fig:topt1}
		\end{minipage}
		\begin{minipage}{0.5\textwidth}
			\centering
			\includegraphics[width=0.8\textwidth]{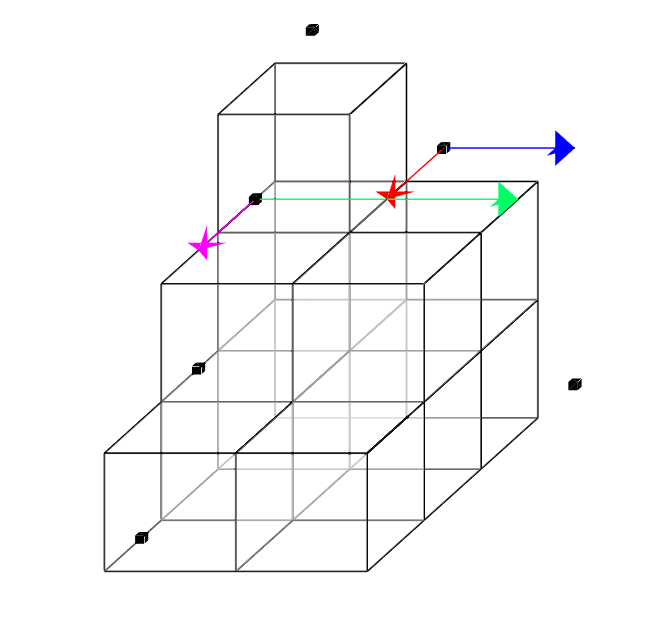}
			\caption{$p$ in pink, $q$ green, $r$ red, $s$ blue}
			\label{fig:topt2}
		\end{minipage}
	\end{figure}
	
	In the second case, assume $\alpha = \alpha_i$ and $\beta=\beta_j$ for some $i,j$.
	To count the morphisms we are interested in from this layer,
	we first count those that take $\alpha$ to $\S$ and then those taking $\alpha$ to zero but $\beta$ to $\S$.
	Let $p$  be the distance in the $x$ direction from $\alpha_i$ to $\alpha_{i+1}$, and $q$  be minimal such that  $y^{q}\alpha/z \in I$,
	let $r$ be the distance in the $x$ direction from $\beta$ to $\alpha$, and let $s$  be the distance in the $y$ direction from $\beta_j$ to $\beta_{j+1}$.
	Then the upward-facing surface of the layer below $\alpha$ and $\beta$ is the disjoint union of a $p \times q$ rectangle and a $r \times s$ one.
	By the same argument as the previous case, arrows taking $\alpha$ to $\S$ that cannot be dragged downwards are those with heads that lie in the
	corner most $p \times q$ rectangle of their layer of $\S$.
	The arrows taking $\alpha$ to zero are arrows from $\beta$ that can be dragged neither downwards nor to have their tail at $\alpha$.
	Arrows for which $x^rf(\beta) \in \S$ can be dragged to $\beta$, so these arrows we count are those with heads that lie in
	the corner most $r \times s$ rectangle of their layer of $\S$.
	Therefore the total number of morphisms counted for $\alpha$ and $\beta$ is
	the number of boxes in the upward-facing surface under $\alpha,\beta$ in each layer below them, which is the number of blocks in $\S$ found directly below this surface. \\
	
	Thus the total count of the morphisms with type three arrows is obtained by counting for every upward-facing surface of $\S$ all the blocks below that surface,
	which will count every block of $\S$ exactly once, giving $d$ morphisms.
	
	Therefore $\dim \Hom(I,R/I)=d+d+d=3d$
\end{proof}

\subsection{Some Examples}
First we look at an example of a point $ I = (x^2,y^2,z^2,xyz) \in \Hilb^{4}(\mathbb{C}^{3})$ that is singular.
Figure \ref{fig:singular} is a diagram with the arrows included.
For a smooth point we would expect that $ \dim \Hom_{R}(I,R/I) = 3 \cdot 4 = 12 $.
But we can clearly count 18 arrows on the diagram, therefore $I$ is not a smooth point.
\begin{figure}[h]
	\begin{minipage}{0.5\textwidth}
		\centering
		\includegraphics[width=\textwidth]{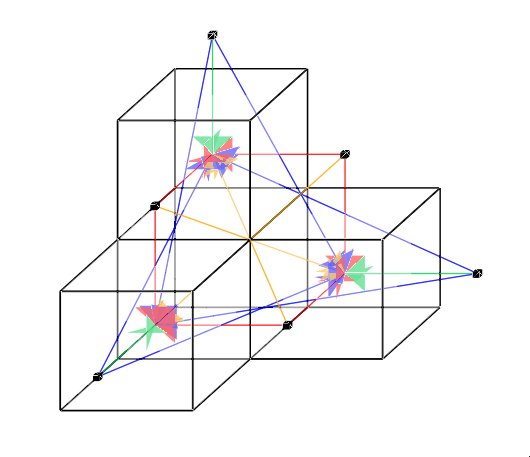}
		\caption{Singular example}
		\label{fig:singular}
	\end{minipage}
	\begin{minipage}{0.5\textwidth}
		\centering
		\includegraphics[width=\textwidth]{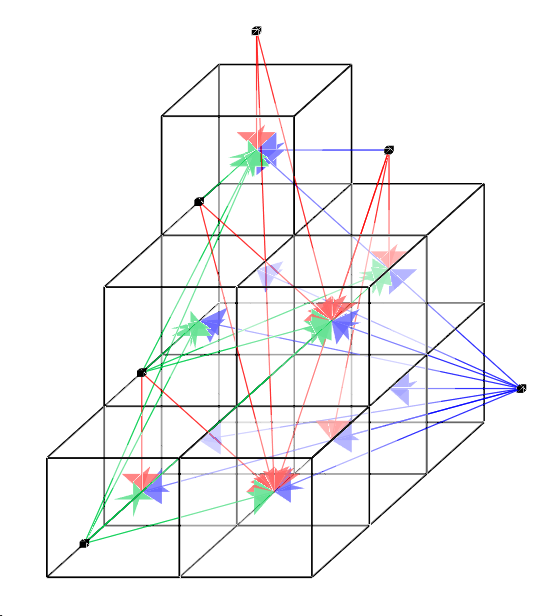}
		\caption{Nonsingular example}
		\label{fig:nonsingular}
	\end{minipage}
\end{figure}

Now we look at an example of a smooth point $ I = (x^3,y^2,x^2z,xz^2,yz^2,z^3) \in  \Hilb^{11}(\mathbb{C}^{3}) $ where we would expect to see $ \dim \Hom_{R}(I,R/I) = 33 $.
Figure \ref{fig:nonsingular} is a diagram with the arrows included.
We can count exactly $33$ arrows, so the point is smooth as expected.

\subsection*{Acknowledgements}

We would like to thank our supervisor - Vladimir Dotsenko - for all his help and assistance during this project. His help made all this possible.


\begin{thebibliography}{99}
	\bibitem{nested_schemes}
	Jan Cheah, \emph{{C}ellular decompositions for nested Hilbert schemes of points},
	Pacific Journal of Mathematics Vol. 183, No. 1, 1998.
	
	\bibitem{schemes}
	Dori Bejleri, \emph{{H}ilbert schemes: geometry, combinatorics, and representation theory}.
\end{thebibliography}
\end{document}